\documentclass[a4paper, leqno, 10pt]{article}

\usepackage{cite}
\usepackage{amsmath}
\usepackage{amsthm}
\usepackage[all]{xy}
\usepackage{amssymb}
\usepackage{mathrsfs}
\usepackage{color}

\newtheorem{thm}{Theorem}[section]
\newtheorem{prop}[thm]{Proposition}
\newtheorem{lem}[thm]{Lemma}
\newtheorem{cor}[thm]{Corollary}
\newtheorem{fct}[thm]{Fact}

\newtheorem{rmk}[thm]{Remark}

\newtheorem{defn}[thm]{Definition}

\theoremstyle{definition}
\newtheorem{rmknit}[thm]{Remark}

\newcommand{\mdls}{\vDash}

\newcommand{\ovl}{\overline}

\newcommand{\ssetq}{\subseteq}

\newcommand{\mR}{\mathbb{R}}

\newcommand{\be}{\begin{equation}}
\newcommand{\ee}{\end{equation}}
\newcommand{\bse}{\begin{subequation}}
\newcommand{\ese}{\end{subequation}}

\providecommand{\Fin}{\mathop{\rm Fin}\nolimits}%
\providecommand{\mod}{\mathop{\rm mod}\nolimits}%
\providecommand{\min}{\mathop{\rm min}\nolimits}%
\providecommand{\sign}{\mathop{\rm sign}\nolimits}%

\renewcommand{\phi}{\varphi}

\title
 {One-basedness and reductions of elliptic curves over real closed fields}

\author{Davide Penazzi}

\begin{document}
 
\maketitle 

\begin{abstract}
Building on the positive solution of Pillay's conjecture we present a notion of ``intrinsic'' reduction for elliptic curves over a real closed field $K$. We compare such notion with the traditional algebro-geometric reduction and produce a classification of the group of $K$-points of an elliptic curve $E$ with three ``real'' roots according to the way $E$ reduces (algebro-geometrically) and the geometric complexity of the ``intrinsically'' reduced curve.

\end{abstract}

\section{Introduction}

Definability in this article is meant in first order logic. Those not familiar with logic can simply consider the class of \emph{definable sets} of a structure $M$ as a class of subsets of $M^n$, for all $n$, determined in a unique way after assigning a language $L_M$, and that is closed under finite union, finite intersection, complementation and projection.

A \emph{definable group} $(G,*)$ in $M$ is a group with a definable underlying set $G\ssetq M^n$ and whose operations $*:G\times G\rightarrow G$ and $^{-1}:G\rightarrow G$ have definable graphs. 

In model theory there exists a notion of an ``infinitesimal subgroup'' $G^{00}$ of a definable group $G$ in a saturated structure $M$.
The group $G^{00}$ is the smallest type-definable bounded-index subgroup of $G$. The motivating example of such group is when $G=([-1,1), +~\mathrm{mod}~2)$ in a real closed field; then $G^{00}$ turns out to be the subgroup of infinitesimal elements around $0$. 

For a large class of structures $G^{00}$ exists; in particular in o-minimal structures we obtain a functorial correspondance $\mathbb{L}:G\rightarrow G/G^{00}$, where $G/G^{00}$ is a real Lie group.
This correspondance is known to preserve many properties of the group and can be thought of a sort of ``model theoretic'' or ``intrinsic'' reduction of the group $G$.

An important question is whether $\mathbb{L}$ preserves the geometric complexity (in the sense of geometric stability theory) of the group $G$.

 The pioneering work of Zilber \cite{Zil} led to a classification of sets in a class of structures called Zariski Geometries: a definable set in a Zariski Geometry either ``resembles'' a pure set, or a vector space or an algebraically closed field.

For o-minimal theories, although they are not Zariski Geometries, a trichotomy classification has been given by Peterzil and Starchenko in \cite{PeSt97}, which roughly states that an o-minimal structure locally resembles either a pure set, or a vector space, or a real closed field.

In this article we work in the o-minimal context, and the concept of a structure having geometric complexity of a vector space is captured by the notion of having a $1$-based theory, following Pillay's work \cite{Pi06a}. An equivalent definition to $1$-basedness is for a structure to have the CF-property (Collapse of Families); this roughly states that given a uniformly definable family of functions, then the germ of such functions at any point can be defined using a single parameter. See \cite{LoPe} for details.

Instead of a theory, we shall analyse the geometric complexity of a definable set or of a type-definable set (i.e. obtained as an infinite, but smaller than the cardinality of the structure, intersection of definable sets) or of a type-definable quotient (i.e. the quotient of a definable set by a type-definable equivalence relation, called a \emph{hyperdefinable set}) induced by the ambient structure in which it lives. 
The method we use is to ``extract'' the induced theory of a definable set and then study $1$-basedness of such theory. 

When the ambient structure is a saturated real closed field $K$, all its definable sets will have the geometric complexity of real closed fields, in particular, also will any definable group $G$.
We can ask if the geometric complexity of hyperdefinable sets $K$ does not always behave so trivially. In particular a good candidate for nonstandard behaviour is the group $G/G^{00}$, where $G$ is definable in a saturated real closed field $K$. We then work in a suitable espansion of $K$ in which $G/G^{00}$ is definable and ``extract'' its theory. 
The general aim of our project is to give a dichotomy classification, \`a la Peterzil-Starchenko of the groups $G/G^{00}$ where $G$ is a $1$-dimensional definable, definably connected, definably compact group in a saturated real closed field $K$. 
Such a project has been initiated in the author's thesis \cite{Penth} and in \cite{Pen1} for some specific groups $G$.

We present in this article an analysis when $G$ is the connected component of an elliptic curve with three ``real'' (meaning in $K$ rather than its algebraic closure) roots.
A first observation is that for elliptic curves $E(M)$ over a valued field $(M,w)$ with discrete valuation group there is an algebro-geometric notion of reduction to a (possibly singular) curve $\tilde{E}(k_w)$ defined over the residue field. There seems not to be such a notion for real closed valued fields, so we need to adapt the algebro-geometric reduction to the context of real closed valued fields (that, we recall, have $\mathbb{R}$ as residue field).

It is natural then to ask if the bad behaviour when an elliptic curve $E(K_w)$ reduces to a singular curve is connected to a loss of structural complexity from the group $E(K_w)^0$ to $E(K_w)^0/E(K_w)^{00}$. This would shed light on what is the relation between the ``intrinsic'' reduction $E(K_w)^0\rightarrow E(K_w)^0/E(K_w)^{00}$ and the ``algebro-geometric'' reduction $E(K_w)\rightarrow \tilde{E}(\mathbb{R})$, and if we can determine model theoretical properties using valuation theoretic notions.  

\vskip 0.3 cm

In the rest of this section we shall describe the setting we work in and the main results obtained in \cite{Pen1}. An outline of the proof of the main theorem is given.

In Section 2 we introduce elliptic curves, the notion of minimal form for an elliptic curve and the definition of algebro-geometric reduction.

In Section 3 we proceed with the study of $1$-basedness when $G=E(K)^0$ where $E$ is an elliptic curve with three ``real'' roots.

In Section 4 we extend the results obtained to truncations of the groups studied in Section 3.

\vskip 0.3 cm

\subsection{Setting and basic facts}

For the rest of the paper $K$ denotes a saturated real closed field, whilst $M$ denotes a saturated o-minimal structure. (Saturated means big enough to find realizations for all consistent types with $<|M|$ parameters in the structure itself.) This allows us to state the results used in full generality.

A definable group $G$ is \emph{definably connected} if there are no proper definable subgroups of finite index, and $G$ is \emph{definably compact} if any definable function from an open interval of the base structure to $G$ has its limit in $G$. The following theorem has been completely proved in \cite{HrPePi06} but is still known as Pillay's conjecture.

\begin{thm}[Pillay's conjecture]

 Given $G$ a definably connected definable group in a saturated o-minimal structure $M$, we have that 
\begin{enumerate}
 \item $G$ has a smallest type-definable subgroup of bounded index $G^{00}$.
\item $G/G^{00}$ is a compact connected Lie group, when equipped with the logic topology.
\item If, moreover, $G$ is definably compact, then the dimension of $G/G^{00}$ (as a Lie group) is equal to the $o$-minimal dimension of $G$.
\item If $G$ is commutative then $G^{00}$ is divisible and torsion-free.
\end{enumerate}
 
\end{thm}

We thus obtain a functor from the category of definable, definably connected, definably compact groups to the category of compact Lie groups: $\mathbb{L}:G\rightarrow G/G^{00}$.

\vskip 0.3 cm

We recall a few facts about o-minimality, in particular the notion of dimension of a definable set in an o-minimal structure; we refer the reader to the book of van den Dries \cite{VdDbook} for an extensive introduction.

Given a structure $M$ and $X\ssetq M^n$ a definable, definably linearly ordered or circularly ordered set, we say that $X$ is \emph{o-minimal}  (resp. \emph{weakly-o-minimal}) if any definable (with parameters from $M$) subset $S\ssetq X$ is a finite union of intervals and points (resp. convex sets).  
We recall that a circularly ordered set is a set equipped with a ternary relation $R(a,b,c)$ meaning that $c$ is after $b$ which is after $a$ clockwise. We then define an \emph{open interval} to be $(a,c)=\{b:R(a,b,c)\}$, and closed intervals and convex sets in the obvious way. For linearly ordered sets we consider as intervals also $(-\infty,a)$ and $(a,\infty)$.

Observe that in the definition of o-minimal sets above when $X=M$ we obtain the usual notion of an \emph{o-minimal structure}.

Basic examples of o-minimal structures are pure linearly ordered dense set without endpoints, such as $(\mathbb{Q},<)$, ordered vector spaces over a field and real closed fields. Real closed fields with a predicate for a convex set, and real closed valued fields are weakly-o-minimal structures.

A well known fact proved by Knight, Pillay and Steinhorn in \cite{KnPiSt} states that if a structure $M$ with language $L_M$ is o-minimal, all structures satisfying the same first order $L_M$-sentences (i.e. all $N$ such that $N\mdls Th(M)$, the theory of $M$) are o-minimal. We can thus say that a theory $T$ is o-minimal if any/all of its models $M\mdls T$ are o-minimal. This is not generally true for o-minimal sets.
 
\vskip 0.3 cm

O-minimal structures carry a notion of dimension:

\begin{defn}
Given a definable set $X$, \[dim(X)=max\{i_1+\dots+i_m|~X ~contains~ an ~(i_1,\cdots,i_m)-cell\}.\]
\end{defn}

An $(i_1,\cdots,i_m)-cell$ is defined inductively by:
\begin{enumerate}
 \item A $(0)-cell$ is a point $x\in M$, a $(1)-cell$ is an interval $(a,b)\in M$.
\item Suppose $(i_1,\dots,i_m)-cells$ are already defined; then an  $(i_1,\dots,i_m,0)-cell$ is the graph of a definable continuous function $f:Y\rightarrow M$, where $Y$ is an $(i_1,\dots,i_m)$-cell; further an $(i_1,\dots,i_m,1)-cell$ is a set of points $\{(\ovl{x},y)| \ovl{x}\in Y, f(\ovl{x})< y <g(\ovl{x})\}$, where $f,g$ are definable continuous functions  $f,g:Y\rightarrow M$, $f<g$ and $Y$ is a $(i_1,\dots,i_m)-cell$.
\end{enumerate}

We say that a definable group $G$ is $n$-dimensional if its underlying set is $n$-dimensional.

\vskip 0.4 cm

Given an o-minimal theory $T$, and a model $M$, with $f(x,\ovl{y})$ a $\emptyset$-definable function in $M$, and $a\in M$, we define an equivalence relation $\sim_a$ on tuples of the same length as $\ovl{y}$ by $\ovl{c}\sim_a\ovl{c}'$ if neither of $f(-,\ovl{c})$, $f(-,\ovl{c}')$ is defined in an open neighbourhood of $a$ or if there is an open neighbourhood $U$ of $a$ such that $f(-,\ovl{c})=f(-,\ovl{c}')$ in $U$.  We call the equivalence class of $\ovl{c}$ the \textit{germ of $f(-,\ovl{c})$ at $a$}, and denote it by $\ovl{c}/{\sim_a}$. 

We say that $T$ is \textit{$1$-based} if in any saturated model $M\mdls T$, for any $a\in M$, for all definable functions $f(x,\ovl{y}):M \times M^n\rightarrow M$, and for any $\ovl{c}\in M^n$ such that $a\notin dcl(\ovl{c})$, we have $\ovl{c}/{\sim_a} \in dcl(a,f(a,\ovl{c}))$ as an imaginary element, i.e., in the appropriate sort of $M^{eq}$: the expansion of $M$ by predicates for all definable quotients.

The basic example of a $1$-based o-minimal theory is the theory of an ordered vector space over a field ($Th(\mathbb{Q},+,0,<)$); an example of non-$1$-based theory is the theory of real closed fields ($Th(\mathbb{R},+,-,\cdot,0,1,<)$).

\vskip 0.4 cm

We define now the structural complexity of a definable set.
Given a definable (infinite) set $S$ in $M$ we can ``extract'' its theory (with all the induced structure from $K$): consider the structure $\mathcal{S}$ whose underlying set is $S$ and work in a language $L_{\mathcal{S}}$ where there is a predicate for every definable (in $M$ and with parameters in $M$) subset of $S^n$ for all $n$. We call the theory $T_S=Th(\mathcal{S})$ the \emph{theory of} $S$ \emph{induced by} $M$. Such a theory is generally hard to study and analyse, since the language will have $|M|$ predicates. 

We obtain a more tame theory for stably embedded sets: a set $S$ is \emph{stably embedded} in $M$ if every definable subset of $S^n$ with parameters in $M$ is definable with parameters from $S$. This implies that $L_{\mathcal{S}}$ need only have predicates for every $\emptyset$-definable subset of $S^n$.

We say that a definable set $S$ is \emph{$1$-based (in $M$)} if the theory $T_S$ is $1$-based.

\vskip 0.2 cm
A basic but fundamental lemma is the following:
\begin{lem}\label{coroneb}
 Given a saturated structure $M$ expanding a field, o-minimal definable sets $X,Y$ definably linearly ordered or circularly ordered, and a definable bijection $\phi:X\rightarrow Y$, then $X$ is (non-) $1$-based if and only if $Y$ is (non-) $1$-based.
\end{lem}
\begin{proof}
 Suppose $X$ is non-$1$-based. The bijection $\phi$ is piecewise order-preserving or order-reversing (otherwise we would be able to define a set that is not a finite union of intervals and points). Suppose $f(x,\ovl{y})$ is a function witnessing non-$1$-basedness in an interval $I$ of $X$. Without loss of generality we can suppose $\phi$ is order-preserving in $I$, then $\phi\cdot f$ witnesses non-$1$-basedness of $Y$.  The other direction is analogous.
\end{proof}
\vskip 0.3 cm

In a real closed field it is well-known that any definable infinite set is non-$1$-based, therefore every definable group $G$ will have the same geometric complexity of a field. We sketch a proof below; the proof uses some results of o-minimality that, although basic, are not recalled in this article. We suggest the book of van den Dries \cite{VdDbook} to the interested reader. 

\begin{fct}\label{defrcf}
 Given a real closed field $K$, any definable infinite set $S\ssetq K^n$ is non-$1$-based.  
\end{fct}
\begin{proof}[Sketch proof:]
 By cell decomposition of $K$ there is a projection $\pi$ on some coordinate of $K^n$ such that $\pi(S)$ contains an interval $I$. Any interval $I\ssetq K$ is in definable bijection with the interval $[0,1)$, and is stably embedded. It suffices, by Lemma \ref{coroneb}, to witness non-$1$-basedness in $[0,1)$. Let $0<a<b<c<1$ be algebraically independent elements such that $a\cdot b +c=d$ is still an element of $[0,1)$. Thus it has $dim(a,b,c,d)=3$ (here it is dcl-dimension). Since $(b,c)/{\sim_a}$ is simply $(b,c)$, if $[0,1)$ were $1$-based, $(b,c)\in dcl(a,d)$ and thus $dim(a,b,c,d)=2$, contradicting $dim(a,b,c,d)=3$. Therefore $[0,1)$ is non-$1$-based, and so is $S$.
\end{proof}

\vskip 0.3 cm

We recall some basics of valuation theory, mantaining the notation of \cite{Pen1}. We denote a \emph{real closed valued field} by $K_w=\big(K,\Gamma_w ,w\big)$, where $K$ is a saturated real closed field with its language, $\Gamma_w$ a divisible abelian ordered group, called the \emph{value group}, with its language, and $w$ a \emph{valuation}, i.e., a surjective map $w:K\rightarrow (\Gamma_w\cup {\infty})$ satisfying the following axioms: for all $x,y\in K$
\begin{enumerate}
 \item $w(x)=\infty \iff x=0$,
\item $w(xy)=w(x)+w(y)$,  
\item $w(x-y)\geq min\{w(x),w(y)\}.$
\end{enumerate}

We denote the \emph{valuation ring} (i.e. the ring $\{x\in K| w(x)\geq 0\}$) by $R_w$, its unique maximal ideal $\{x\in K| w(x)> 0\}$ (the \emph{valuation ideal}) by $I_w$, $k_w=R_w/I_w$ the \emph{residue field}; we recall moreover that the value group $\Gamma_w$ is $K^*/(R_w\setminus I_w)$.

When the valuation ring is $\Fin$: the convex hull of $\mathbb{Q}$ in $K$, we call the valuation the \emph{standard valuation} and denote it by $v$; the corresponding real closed valued field is $M_v$. The valuation ideal is $\mu$, the infinitesimal neighbourhood of $0$. The standard residue field, $k_v$, is $\mathbb{R}$, and the projection $\Fin\rightarrow \mathbb{R}$ is called the \emph{standard part map}.

\vskip 0.2 cm

We can obtain a real closed valued field from a real closed field via a particular kind of Dedekind cut, called a valuational cut: a \emph{valuational cut} in a structure $(M,+,0,<,\dots)$ expanding an ordered group is a cut $\alpha$ such that there exists $\epsilon\in M$, $\epsilon>0$, for which $\alpha+\epsilon=\alpha$. By Theorem 6.3 of \cite{MMS}, if $M$ is a weakly o-minimal expansion of an ordered field with a definable valuational cut, then $M$ has a nontrivial definable convex valuation.

\vskip 0.2 cm

We define the \emph{open balls} $B_{>\gamma}(a)=\{x\in K|w(x-a)>\gamma\}$ and \emph{closed balls} $B_{\geq\gamma}(a)=\{x\in M|w(x-a)\geq\gamma\}$, where $\gamma\in \Gamma_w$ and $a\in K$. A simple remark is:

\begin{rmk}\label{skelk}
There is a definable field isomorphism $B_{\geq\gamma}(0)/B_{>\gamma}(0)\cong k_w$ for any $\gamma\in \Gamma_w$
\end{rmk}
Clearly the map $f:B_{\geq\gamma}(0)\rightarrow B_{\geq0_{\Gamma_w}}(0)$, sending $x\mapsto \frac{x}{u}$, where $u\in K $ such that $w(u)=\gamma$, is well defined in the quotients $B_{\geq\gamma}(0)/B_{>\gamma}(0)\rightarrow B_{\geq 0_{\Gamma_w}}(0)/B_{> 0_{\Gamma_w}}(0)=k_w$ and is a field isomorphism.

\begin{rmknit}\label{mellrem}

In \cite{Me06}, Mellor proved that every definable subset of $\Gamma_w^n$ (resp. $k_w^n$) definable with parameters from $M_w$ in its valued field language is definable with parameters from $\Gamma_w$ (resp. $k_w$) in its ordered group (resp. ordered field) language. This implies the following fact
 
\end{rmknit}

\begin{fct}\label{onebgammak}
 $Th(\Gamma_w)=Th(\mathbb{Q},+,0,<)$, and therefore $\Gamma_w$ is $1$-based in $M_w$. Analogously $Th(k_w)=Th(\mathbb{R},+,\cdot,0,1,<)$, and therefore $k_v$ is non-$1$-based in $M_w$.
\end{fct}

Given a group equipped with a linear order $G=(G,*,<)$, a \emph{truncation} of $G$ by an element $a$ is the group $\left(\left[a^{-1},a\right),*\mod a^2\right)$, where the operation \\$*_{\mod a^2}$ is defined as follows: 
\[b *_{\mod a^2} c= \left\{\begin{array}{ll}
 b * c &~if~  a^{-1}<b * c <a \\
b * c * a^{-1} & ~if ~b * c >a \\
b * c * a & ~if ~b * c <a^{-1} ~~~.\\
\end{array}\right. 
\]

The linear order $<$ of $G$ naturally induces a circular ordering on the truncation.

\vskip 0.3 cm

In \cite{Pen1} the following theorem is proved:

\begin{thm}
Given a definable, definably compact, definably connected, one-dimensional (in the o-minimal sense) group $G$ in a saturated real closed field $K$, if $G$ is an additive truncation, a small multiplicative truncation, i.e., $G=\left(\left[b^{-1},b\right),*~mod~b^2\right)$, with $v(b)=0$, or a truncation of $SO_2(K)$, $G/G^{00}$ is non-$1$-based in the expansion of $K$ by a predicate for $G^{00}$.

If $G$ is a big multiplicative truncation, i.e., $G=\left(\left[b^{-1},b\right),*~mod~b^2\right)$, with $v(b)<0$, the group $G/G^{00}$ is $1$-based in the expansion of $K$ by a predicate for $G^{00}$.
\end{thm}

\vskip 0.3 cm
\subsection{Main theorem and outline of its proof}

The rest of the article is devoted to prove Theorem \ref{thmpaper} below, the outline of the proof is given here, and the details of the proof are carried out in the following sections:

\begin{itemize}
 \item Given an elliptic curve $E$ over $K$, we shall define a notion of minimal form of an elliptic curve, and for curves in minimal form we define three kinds of reductions of their $K$-points.
\end{itemize}

This is done in Section 2.

\begin{itemize}
 \item When we consider $G$ to be the semialgebraic connected component of the $K$-points of an elliptic curve in minimal form over $K$: $E(K)^0$, or $G$ is a truncation of $E(K)^0$, then its unique minimal, bounded index, type-definable subgroup $G^{00}$ determines a valuational cut on $K$.
\end{itemize}

This is proven at the beginning of Section 3 and in Section 4.

\vskip 0.2 cm
We denote the  structure $(K,G^{00},\dots)^{eq}$ by $K'$. In $K'$ the cut above becomes definable and it determines a valuation $w$ on $K$. So, given a group $G$ as above, we canonically determine (definably) in $K'$ a value group $\Gamma_w$ and a residue field $k_w$; therefore $K'$ will be interdefinable with a real closed valued field $K_w^{eq}$, and we shall use this identification throughout the article.

The group $G/G^{00}$ is thus a definable set in $K'$ and it now makes sense to ask whether it is $1$-based or not. We show, case by case, that:

\begin{itemize}
 \item  The group $G/G^{00}$ is in definable bijection with a definable group whose underlying set is a subset of $\Gamma_w^n$ for some curves and of $k_w^n$ for other curves (see the points 1.3 and 2.3 of Theorem \ref{thmpaper} below for details). 
\end{itemize}  

This is proven in Sections 3.1, 3.2 and 4.

We shall identify $G/G^{00}$ with the group it is in definable bijection with using Lemma \ref{coroneb}.

By Theorem $2$ of \cite{HaOn09}, $G/G^{00}$ is stably embedded in $K'$, i.e., every subset of $(G/G^{00})^n$ definable with parameters from $K'$ is definable with parameters from $G/G^{00}$.

This and Remark \ref{mellrem} imply that $T_{G/G^{00}}$ as a definable set in $K'$ equals the theory $T_{G/G^{00}}$ as a definable set of $\Gamma_w$ (resp. $k_w$) seen as a structure on its own, i.e. as an ordered vector space (resp. a real closed field).

Thus, using Fact \ref{onebgammak}, $T_{G/G^{00}}$ is $1$-based if and only if $G/G^{00}$ is in definable bijection with a definable group whose underlying set is a subset of $\Gamma_w^n$.

The full statement of the main theorem is then the following:

\begin{thm}\label{thmpaper}
Given the group $G=E(K)^0$, or $G$ a truncation of $E(K)^0$, where $E$ is an elliptic curve with three ``real'' roots, over a saturated real closed field $K$, the structure $K'$ obtained by adding a predicate for $G^{00}$ to $K$ is interdefinable with a real closed valued field $K_w$.
\vskip 0.2 cm
There are two possible behaviours, either the following set of conditions hold:
\begin{enumerate}
 \item[1.1] The group $G/G^{00}$ is $1$-based in $K'$.

\item[1.2] The group $G/G^{00}$ is in definable bijection with a definable group in $K'$ whose underlying set is a subset of $\Gamma_w^n$.

\item[1.3] 
\begin{itemize}
 \item Either $G=E(K)^0$ and $E$ has split multiplicative reduction, or
\item $G$ is the truncation of $E(K)^0$ by a point $P$ with infinitesimal projection on the $x$-axis, where $E$ is an elliptic curve with split multiplicative reduction.
\end{itemize}
\end{enumerate} 

Or the following conditions hold:
 \begin{enumerate}
 \item[2.1] The group $G/G^{00}$ is non-$1$-based in $K'$.

\item[2.2] The group $G/G^{00}$ is in definable bijection with a definable group in $K'$ whose underlying set is a subset of $k_w^n$.

\item[2.3] 
\begin{itemize}
 \item Either $G=E(K)^0$, or $G$ is a truncation of $E(K)^0$, where $E$ has good or nonsplit multiplicative reduction, or
\item $G$ is the truncation of $E(K)^0$ by a point $P$ with non-infinitesimal projection on the $x$-axis , where $E$ is an elliptic curve with split multiplicative reduction.
\end{itemize}
\end{enumerate} 
\end{thm}

\begin{rmk}
Conditions 1.2 (resp. 2.2) above can be stated in model theoretic terms as: the group $G/G^{00}$ is internal to $\Gamma_w$ (resp. $k_w$) in $K'$.
\end{rmk}

\section{Elliptic curves}
An introduction to the theory of elliptic curves can be found in the book of Silverman, \cite{Silbook}. Here we briefly recall the main notions and define the algebro-geometric reduction for curves defined in a real closed field.

An \textit{elliptic curve} over a field $F$ is a nonsingular one-dimensional projective curve defined by an equation of the form \[E:~ y^2+a_1xy+a_3y=x^3+a_2x^2+a_4x+a_6,\] where $a_1,\dots, a_6\in F$, plus a point at infinity, denoted by $O$. Given a field $K$, $E(K)=\{(x,y)\in K^2|y^2+a_1xy+a_3y=x^3+a_2x^2+a_4x+a_6\}$ is the \emph{set of $K$-points of $E$}.

When we work in the projective space we define it by $ZY^2+a_1XYZ+a_3YZ^2=X^3+a_2X^2+a_4XZ^2+a_6Z^3$, and the point at infinity is $O=[0:1:0]$.

We can endow $E(K)$ with a group structure, whose identity is $O$.  
Any line will intersect an elliptic curve at precisely three points (also $O$ is a point). Given points $P,Q$, the line through $P$ and $Q$ (or the tangent line if $P=Q$) intersects $E$ at the point $R$. The line between $R$ and $O$ will again intersect $E$ at one point, which we call $R'$. We then define $P\oplus Q$ to be $R'$. We denote the inverse of a point $P$ by $\ominus P$.

There exists also an algebraic definition for this operation, which we will state later, after simplifying the form of the curve.

As any abelian group, $E$ is also a $\mathbb{Z}$-module, with scalar operation denoted by $[m]P$.

Working with a real closed field $K$ and an elliptic curve $E$ defined over $K$, $E(K)$ is a topological group, but with the usual topology of $K$ it is totally disconnected. So, instead of considering the usual connected component of $E(K)$, we consider its \textit{semialgebraic (definable) connected component} $E(K)^0$.

\vskip 0.3 cm
In this article we view $(E(K)^0,\oplus)$ as living in two different categories: model theoretically as a definable group in $K$, to which we can apply the functor $\mathbb{L}$ described in Pillay's conjecture; and algebro-geometrically as the $K$-points of a curve, to which we can apply the reduction map.

Whilst the model theoretic functor is defined intrinsically and can be applied to curves in any form, the reduction map depends on how $E(K)^0$ sits in the ambient space. We need thus to determine a minimal form of the elliptic curve.

To ease the further computations the most obvious choice is to consider the curve in its Legendre form $y^2=x(x-1)(x-\lambda)$, where $\lambda\in K^{\mathrm{alg}}=K[i]$ and $\lambda \neq 0,1$, to ensure non-singularity. If $\lambda\in K$ we say that the elliptic curve has three ``real roots'', where by \textit{root} we mean a point in which $E(K)$ intersects the $y=0$ line. We only discuss curves with real roots in this article.

A translation and a homothety transform our curve into $y^2=x(x+1)(x+\epsilon)$, with $0<\epsilon<1$. Such a curve is said to be in \emph{minimal form} in an analogue for real closed fields of the minimal form for local fields defined in  Proposition 1.3, Chapter $VII$ of \cite{Silbook}.
We can explicitly express the sum and the doubling formulae for curves in minimal form in a relatively simple way:

\be\label{sum}x_{P\oplus Q}=\left(\frac{y_Q-y_P}{x_Q-x_P}\right)^2-(1+\epsilon)-x_Q-x_P, \ee

\be\label{duplicrcurfor}x_{[2]P}=\frac{(x_P^2+\epsilon)^2}{4x_P(x_P+1)(x_P-\epsilon)}, \ee

where a point $P$ is denoted by $P=(x_P,y_P)$.

\subsection{Algebro-geometric reductions}

An important tool in the arithmetic study of elliptic curves defined over local fields is the notion of reduction over the residue field. This topic is developed in Chapter VII of \cite{Silbook}. We present here a description of this tool, adapted to the context of real closed fields.

\vskip 0.2 cm
We suppose that $E$ is an elliptic curve in minimal form defined over a saturated real closed field $K$, and equip $K$ with the standard valuation. 
When we project the $K$-points $E(K)$ of the elliptic curve onto the standard residue field we obtain a curve $\tilde{E}(\mathbb{R})$ which is easier to study. The definition of this operation is delicate and requires some care.

We define the reduction $\tilde{E}$ of a curve $E: y^2=x(x+1)(x+\epsilon)$ to be the curve over $k_v$ defined by $y^2=x(x+1)(x+\mathrm{st}(\epsilon))$, with $st:\Fin\rightarrow \mathbb{R}$ the standard part map.


This gives us a reduction map \[\begin{array}{rl} E(K)&\rightarrow \tilde{E}(\mathbb{R})\\
                                 P&\mapsto \tilde{P}
                                \end{array}\]

defined as follows: given a point $P=(x_P,y_P) \in E(K)$ we rewrite it in homogeneous coordinates: $P=[x_P;y_P;1]$. This clearly can always be rewritten with coefficients in $\Fin$ by multiplying the coordinates by a scalar $\lambda$: $P=[x';y';z']$, with at least one coefficient with valuation $0$. We can now project the coordinates onto the residue field, and $P$ reduces to $\tilde{P}=[st(x');st(y');st(z')]$. We multiply back by $\lambda^{-1}$ to obtain $\tilde{P}=[\lambda^{-1}\left(st(x')\right);\lambda^{-1}\left(st(y')\right);\lambda^{-1}\left(st(z')\right)]$.

 In affine coordinates it is then simply  
\[ \left\{\begin{array}{ll}
   \tilde{P}=(st(x_P),st(y_P)) & ~if~ x,y\in \Fin\\
\tilde{P}=O & ~if~ x,y\notin \Fin.
  \end{array}\right. \]

This operation, however, is not harmless: $\tilde{E}(\mathbb{R})$ may not longer be an elliptic curve, and it could have singularities. The set of nonsingular points of $\tilde{E}(\mathbb{R})$ forms a group, defined over $\mathbb{R}$, denoted by $\tilde{E}_{ns}(\mathbb{R})$.  

We define two subsets of $E(K)$ depending on how the curve reduces: 
\be E_0(K)=\{P\in E(K): \tilde{P}\in\tilde{E}_{ns}(\mR)\},\ee
i.e., the set of all points of $E$ whose reduction is nonsingular, and

\be E_1(K)=\{P\in E(K): \tilde{P}= \tilde{O}\} ~(=\{P\in E(K)|v(x_P)<0\}),\ee
i.e., the set of all points whose reduction is the identity of $\tilde{E}_{ns}(\mR)$.

Having chosen a minimal form for the elliptic curve such notions are well defined.

\vskip 0.3 cm

A useful proposition is the following:

\begin{prop}\label{e01iso}
There is a group isomorphism $E_0(K)/E_1(K)\cong \tilde{E}_{ns}(\mathbb{R})$.
\end{prop}
\begin{proof}
After observing that a real closed valued field satisfies Hensel's Lemma (this is folklore, a proof of this fact is in Theorem 4.3.7 of \cite{Engbook}), it is sufficient to follow the proof of Proposition 2.1 of Chapter VII of \cite{Silbook}. 
\end{proof}

\vskip 0.2 cm
We easily compute the possible reductions of curves of the form $E:y^2=x(x+1)(x+\epsilon)$, with $0<\epsilon<1$, over the reals:
\begin{rmk}\label{kindred}

We obtain three kinds of curves:
\begin{enumerate}
 \item Good reduction curves: if $v(\epsilon)=0$ and $v(\epsilon-1)=0$, this imples that the standard part of the root $(\epsilon,0)$ does not equal the standard part of any of the other roots, and therefore the reduced curve is nonsingular.
\item Non-split multiplicative reduction curves: if $v(\epsilon-1)>0$, this implies that the root $(\epsilon,0)$ is sent by the standard part map to the root $(-1,0)$, and therefore the reduced curve has a complex node.
\item Split multiplicative reduction curves: if $v(\epsilon)>0$, this implies that the root $(\epsilon,0)$ is sent by the standard part map to the root $(0,0)$, and therefore the reduced curve has a real node.
\end{enumerate}

\end{rmk}

\section{Case study}

To study the relation between intrinsic and algebro-geometric reductions we need to be able to determine $G^{00}$. Proposition $2$ of \cite{Raz} tells us that $G$ is definably circularly ordered. Moreover we can define a dense linear orientation on $G\setminus \{\mathrm{point}\}$. Since $G^{00}$ is a neighbourhood of the identity we choose to remove the ``farthest'' point from $O$: the $2$-torsion point $T_2$, and obtain the orientation $\lhd$ on $G\setminus \{\mathrm{T_2}\}$. By the proof of Proposition $3.5$ of \cite{Pi04d}, $G^{00}$ is bounded by the torsion points of $G$, namely, it is type-defined by:

 \be\label{g00tors} G^{00}=\bigcap_{n\in\omega}\Big\{P| \forall T \left[(T \rhd O\wedge [n]T=O)\rightarrow \ominus T \lhd P\lhd T\right]\Big\},\ee

\begin{defn}\label{boundtors}
We call a \emph{bounding sequence of torsion points} a subsequence $(T_{i_n})_{n\in \omega}$ of the sequence \\$(T_n)_{2<n<\omega}$ of torsion points such that $[n]T_n=O$ (i.e., $T_n$ is an $n$-torsion point), and there is no $n$-torsion point $T$ such that $O\lhd T\lhd T_n$.
\end{defn}

A bounding sequence of torsion points $(T_{i_n})_{n\in \omega}$  easily determines $G^{00}$:
 \be\label{boundinggoo} G^{00}=\bigcap_{2<n<\omega}\{T|\ominus T_{i_n}\lhd T\lhd T_{i_n}\}.\ee

\vskip 0.3 cm
As discussed in the previous section, we suppose from now on that $E$ is $y^2=x(x+1)(x+\epsilon)$, with $0<\epsilon<1$.

Since the duplication formula allows us determine the $2^n$-torsion points, we shall use the bounding sequence: $(T_{2^{n}})_{n>1}$ to compute $G^{00}$. Recall also that $y_{T_{2^{n}}}>0$, and thus $y_{\ominus T_{2^{n}}}<0$.

It is easy to compute directly $T_4$, considering the tangent to the curve passing by $(0,0)$ we determine that $x_{T_4}=\sqrt{\epsilon}$.

For the other points of the bounding sequence we shall just consider an approximation given by taking the standard valuation of their $x$-coordinate. In particular $v\left(x_{T_4}\right)=\frac{1}{2}v(\epsilon)$. The choice of the $4$-torsion points as our starting point for the bounding sequence is no coincidence: for points $P,Q$ such that $T_4\lhd P,Q \lhd O$, the operations of sum and formal multiplication respect the orientation. We thus deduce the convenient inequalities: 
\be\label{formul}v\left(x_{[2]P}\right)\geq v\left(x_{P}\right)\ee and \be\label{sumval}v\left(x_{P\oplus Q}\right)\geq v\left(x_{P}\right),v\left(x_{Q}\right).\ee

\vskip 0.3 cm
We recall and shall often use without further mention the following fact: if $v(a)\neq v(b)$ or $\sign(a)=\sign(b)$, then $v(a+b)= \min\{v(a),v(b)\}$.

\begin{lem}\label{e00eq}
Let $E$ be a curve in the form $y^2=x(x+1)(x+\epsilon)$, with $\epsilon>0$, and $G=E(K)^0$. Then $G^{00}=\bigcap_{n\in \omega}\left\{P\in G|v\left(x_P\right)<\frac{1}{n}v(\epsilon)\right\}$.
\end{lem}

\begin{proof}

It is sufficient to prove that, for $n\geq 2$, $v\left(x_{T_{2^{n-1}}}\right)=\frac{1}{2}v\left(x_{T_{2^n}}\right)$, for $T_{2^n}$ a bounding sequence of torsion points. In fact by (\ref{g00tors}), and by symmetry of the curve with respect to the $x$-axis, \[ G^{00}=\bigcap_{n\in \omega}\{P|v(x_P)\leq v\left(x_{T_{2^n}}\right)\}.\]


We have two cases:

\begin{enumerate}
 \item If $v(\epsilon)= 0$, by induction we may assume $v\left(x_{T_{2^{n-1}}}\right)=0$, by the observations above $0=v\left(x_{T_{2^{n-1}}}\right)\geq v\left(x_{T_{2^n}}\right)\geq 0$, so $v\left(x_{T_{2^n}}\right)=0$, and thus the equality above is verified.

For this case we also need to check that the torsion points have cofinal projection in $\Fin$, but $x_{T_{2^{n-1}}}=\frac{1}{4}\frac{\left(x_{T_{2^{n}}}^2-\epsilon\right)^2}{x_{T_{2^{n}}}\left(x_{T_{2^{n}}}+1\right)\left(x_{T_{2^{n}}}+\epsilon\right)}<\frac{x^4_{T_{2^{n}}}}{4x^3_{T_{2^{n}}}}=\frac{1}{4}x_{T_{2^{n}}}$.

From which $x_{T_{2^{n}}}>\frac{1}{4^{n-2}}x_{T_4}=\frac{1}{4^{n-3}}\epsilon$. So, for each $m\in \Fin$, there is $n$ such that $x_{T_{2^{n}}}>m$, i.e., the bounding sequence of torsion points has cofinal projection in $\Fin$.

\item If $v(\epsilon)> 0$, using the duplication formula we get:\\
 $v\left(x_{T_{2^{n-1}}}\right)=v\left(\frac{1}{4}\frac{\left(x_{T_{2^n}}^2-\epsilon\right)^2}{x_{T_{2^n}}\left(x_{T_{2^n}}+1\right)\left(x^2_{T_{2^n}}+\epsilon\right) }\right)=$
$2v\left(x^2_{T_{2^n}}-\epsilon\right)-v\left(x_{T_{2^n}}\right)-v\left(x_{T_{2^n}}+1\right)-v\left(x_{T_{2^n}}+\epsilon\right)=$

$\big($ since $v\left(x_{T_{2^n}}+1\right)=0$ and $v\left(x_{T_{2^n}}+\epsilon\right)=v\left(x_{T_{2^n}}\right)$ $\big)$\\

$=2v\left(x^2_{T_{2^n}}-\epsilon\right)-2v\left(x_{T_{2^n}}\right)$.

Observe that $v\left(x^2_{T_{2^n}}-\epsilon\right)=v\left(x^2_{T_{2^n}}\right)$, in fact otherwise $v\left(x_{T_{2^n}}\right)=\frac{1}{2}v(\epsilon)$ and so $\frac{1}{2}v(\epsilon)=v\left(x_{T_{4}}\right)> 2v\left(x^2_{T_{8}}\right)-2v\left(x_{T_{8}}\right)=2v\left(x_{T_{8}}\right)=v(\epsilon)$, contradicting $v(\epsilon)>0$.

Then we have $v\left(x_{T_{2^{n-1}}}\right)=2v\left(x^2_{T_{2^n}}\right)-2v\left(x_{T_{2^n}}\right)=2v\left(x_{T_{2^n}}\right)$ and we have proved the lemma.

\end{enumerate}

\end{proof}

We observe that the projection $\alpha$ onto the $x$-axis of $G^{00}$ is a valuational cut. We recall that $\alpha$ is valuational if there exists an $\epsilon\in K^{>0}$ such that $\alpha+\epsilon=\alpha$. This is witnessed by the same $\epsilon$ defining $G$. There is therefore a unique valuation $w$, not necessarily the standard one, associated to $G^{00}$, definable in $K'=(K,G^{00},\dots)^{eq}=K^{eq}_w$.

\vskip 0.2 cm
We now study which elliptic curves $G=E(K)^0$ determine $G/G^{00}$ $1$-based in $K'$, and relate the map $G\rightarrow G/G^{00}$ to the behaviour of $E(K)^0$ when reduced over the standard residue field. 

We have three possible kinds of reduction; see Remark \ref{kindred}.

\vskip 0.3 cm

\subsection{The good reduction and the non-split multiplicative reduction cases}

These are the cases of a curve $E:y^2=x(x+1)(x+\epsilon)$ in minimal form, with $v(\epsilon)=0$. We show that for such cases the intrinsic and the algebro-geometric reductions coincide (at least when we consider the semialgebraic connected component $E(K)^0$).

In fact the algebro-geometric reduction leads to the curve $\tilde{E}(\mathbb{R}): y^2=x(x+1)(x-st(\epsilon))$.
\vskip 0.2 cm

Clearly then $E(K)^0=E_0(K)^0$, and, by Lemma \ref{e00eq}, \[E_1(K)^0=\{P\in E(K)| v(x_P)<0\}=G^{00}.\] 

This, together with Proposition \ref{e01iso}, implies that

\be\label{gored} G/G^{00}=E(K)^0/E(K)^{00}=E_0(K)^0/E_1(K)^0\cong \tilde{E}^0(\mathbb{R}).\ee

\vskip 0.3 cm

We add to $K$ a predicate for $G^{00}$: let $K'=(K,G^{00},\dots)^{eq}$. The valuational cut determined by $G^{00}$ induces the standard valuation on $K'$: we can in fact define in $K'$ the sets $\Fin$ and $\mu$: \be \Fin=\left\{x\in K| \exists y\in K \Big((x,y)\notin G^{00} \wedge (-x,y)\notin G^{00} \Big)\right\},\ee
\be \mu=\left\{x\in K| x^{-1}\notin \Fin\right\}.\ee

Clearly in the standard real closed valued field (with symbols for the imaginaries) $K_v=(K,\Fin,\mu,v,\dots)^{eq} $ the set $G^{00}$ is definable, so $K'$ is interdefinable with $K_v^{eq}$.

Moreover $G/G^{00}$ is definably isomorphic in $K'$ to the group $E^0(\mathbb{R})$, that is a definable group with underlying set in $k_v$. By Fact \ref{onebgammak}, $k_v$ is non-$1$-based in $K'$ and by Lemma \ref{coroneb} and Fact \ref{defrcf} also $G/G^{00}$ is non-$1$-based in $K'$.

We have therefore proved the following lemma:
\begin{lem}\label{goodnonsp}
 Given an elliptic curve $E$ in minimal form, and such that $E(K)$ has good or nonsplit multiplicative reduction, the group $G/G^{00}$, where $G=E(K)^0$, is non-$1$-based in $K'=(K,G^{00},\dots)^{eq}$ and is definably isomorphic to a group with underlying set in $k_v$, the residue field of the standard real closed valued field interdefinable with $K'$.
\end{lem}

\vskip 0.2 cm

Whilst in the good reduction case (i.e. when $v(\epsilon-1)=0$) the definable (in $K'$) isomorphism of groups $ G/G^{00}\cong \tilde{E}^0(\mathbb{R})$ extends naturally to an isomorphism $E(K)/E(K)^{00}\cong\tilde{E}(\mathbb{R})$ with the reduced curve; the difference between intrinsic and algebro-geometric reductions unveils when we look at Lie structure of the whole curve when we have nonsplit multiplicative reduction.
In the non-split multiplicative reduction case the algebro-geometric reduction leads to a singular curve with a ``complex node'' at the point $(-1,0)$.
By Exercise 3.5, page 104 of \cite{Silbook}, $\tilde{E}(\mathbb{R})^0\cong SO_2(\mathbb{R})$ as a Lie group.
 
So, applying the algebro-geometric reduction to $E(K)$, we obtain a connected component isomorphic to $SO_2(\mathbb{R})$ and an isolated point $(-1,0)$, whereas the image $E(K)/E(K)^{00}$ of $E(K)$ under the functor $\mathbb{L}$ is still a nonsingular curve, with the two connected components in bijection and therefore both isomorphic to $SO_2(\mathbb{R})$.

\vskip 0.3 cm
\subsection{The split multiplicative reduction case}

This is the case of a curve $E:y^2=x(x+1)(x+\epsilon)$ where $v(\epsilon)>0$; the algebro-geometric reduction of $E(K)$ is then a curve with a singularity, more precisely a real node, at $(0,0)$.

\vskip 0.2 cm
We denote by $H$ the group $\left([\epsilon,\frac{1}{\epsilon}),*~mod~\epsilon^2\right)$ (a ``big'' truncation of the multiplicative group by $\epsilon$). Theorem 4.10 of \cite{Pen1} states that the group $H/H^{00}$ is $1$-based in $K_{H^{00}}=(K,H^{00},\dots)^{eq}$. 
To obtain $1$-basedness for $G/G^{00}$ in $K'=(K,G^{00},\dots)^{eq}$ from the known case of the ``big'' multiplicative truncation, it will suffice, by Lemma \ref{coroneb}, to show that $K_{H^{00}}$ is interdefinable with $K'$, and to find a definable bijection $f:G/G^{00}\rightarrow H/H^{00}$.

\vskip 0.2 cm

We denote by $P$ a point in $G$ and by $P_\sim$ the class in $G/G^{00}$ of which it is a representative. Analogously we denote by $x$ an element of $H$ and by $x_\sim$  an element in $H/H^{00}$. 

We firstly define a map $f_*: G\rightarrow H$ as follows:
\[ f_*(P)=\left\{\begin{array}{cl}
               1& ~\mathrm{if} ~x_P\geq 1,\\
\\
               \left(\frac{1}{x_P}\right)& ~\mathrm{if}~ y_P\geq0 \wedge \epsilon< x_P< 1 ,\\
\\
 x_P& ~\mathrm{if}~ y_P<0\wedge \epsilon<x_P < 1,\\

\epsilon& ~\mathrm{if}~x_P\leq \epsilon.

              \end{array}\right.\]

We prove that $f_*$ induces a well-defined bijection $f:G/G^{00} \rightarrow H/H^{00}$ on the quotients. 
Due to the definition of $f_*$ it is necessary to consider separately the cases of $G^{00}$ and of $(T_2)_{\sim}$.

\begin{lem}
The map $f_*$ sends $G^{00}$ to $H^{00}$, and so $f(O_{\sim})=1_{\sim}$.
\end{lem}
\begin{proof}
 We recall Lemma \ref{e00eq}:\\ $$G^{00}= \bigcap_{n\in\omega}\left\{P|~ v(x_P)<\frac{1}{n}v(\epsilon)\right\}.$$ It easy to see that $H^{00}=\bigcap_{n\in\omega}\left\{x|~\epsilon<x^n<\frac{1}{\epsilon}\right\}=\bigcap_{n\in\omega}\left\{x|~ \left| v(x) \right| <\frac{1}{n}v(\epsilon)\right\}$. Thus $f_*(G^{00})=H^{00}$, and then $f(O_{\sim})=1_{\sim}$.
\end{proof}

\vskip 0.2 cm

We characterize $(T_2)_\sim$ via the valuation of the projection of its points on the $x$-axis.

\begin{lem}\label{classa}
 We have  $(T_2)_\sim=\bigcap_{n\in\omega}\left\{P\in G|v(x_P)\geq \frac{n-1}{n}v(\epsilon)\right\}.$
\end{lem}
\begin{proof}
By definition $P\in (T_2)_\sim$ if and only if $ P\ominus {T_2}\in G^{00}$ if and only if $v(P\ominus {T_2})<\frac{1}{n}v(\epsilon)$, for all $n$.

Then, using (\ref{sum}), $v(x_{P\ominus {T_2}})=v\left(\frac{y_P^2}{x_P^2}-1-\epsilon-x_P\right)=$\\$=v\left(\frac{(x_P+1)(x_P+\epsilon)}{x_P}-1-\epsilon-x_P\right)= v\left(x_P^2+x_P+\epsilon x_P+\epsilon-x_P-\epsilon x_P-x_P^2\right)-2v(x_P)=v(\epsilon)-v(x_P)$.  So  $v(x_{P\ominus {T_2}})<\frac{1}{n}v(\epsilon)$, for all $n$, if and only if $v(x_P)\geq \frac{n-1}{n}v(\epsilon)$, for all $n$.
\end{proof}

In $H/H^{00}$ the class of the $2$-torsion $h_2=\epsilon$ is \[(h_2)_{\sim}= \left\{x\in H| |v(h)|\geq \frac{n-1}{n}v(\epsilon)\right\}.\] The proof of the following lemma is now immediate.

\begin{lem}
 The map $f$ sends $(T_2)_\sim$ to $(h_2)_{\sim}$. 
\end{lem}

\vskip 0.3 cm

We want to prove for all the other cases that the map $f$ is well-defined.

\begin{thm}\label{defbij}
The map $f$ is a well-defined function $G/G^{00}\rightarrow H/H^{00}$.
\end{thm}
\begin{proof}

Let $P,Q\in P_\sim$, then $P\ominus Q\in G^{00}$, i.e., $v(x_{P\ominus Q})<\frac{1}{n}v(\epsilon)$, for all $n$.
Our aim is to prove that $f_*(P)\sim f_*(Q)$: i.e., $f_*(P) f_*(Q)^{-1}\in H^{00}$. Notice that we already proved this for the class of ${T_2}$ and for $G^{00}$, we shall then suppose $P,Q\notin (T_2)_\sim$, and $P,Q\notin G^{00}$, so we have, by symmetry of the elliptic curve and the lemmas above, $\sign(y_P)= \sign(y_Q)$ and $v(\epsilon)>v(x_Q),v(x_P)>\frac{1}{m}v(\epsilon)$ for some $m\in \mathbb{N}$.

Suppose then that for all $n$ we have $\frac{1}{n}v(\epsilon)>v\left(x_{P\ominus Q}\right)$. Using the addition formula (\ref{sum}) and the fact that $x_{\ominus Q}=x_Q$ and $y_{\ominus Q}=-y_Q$ we have $v\left(x_{P\ominus Q}\right)=v\left(\frac{(y_P+y_Q)^2}{(x_P-x_Q)^2}-\epsilon-1-x_P-x_Q\right)=v\big(x_P(x_P+1)(x_P+\epsilon)+x_Q(x_Q+1)(x_Q+\epsilon)+2y_Py_Q-\epsilon x_P^2-\epsilon x_Q^2+2\epsilon x_Px_Q-x_P^2-x_Q^2-2x_Px_Q-(x_P+x_Q)(x_P-x_Q)^2)-2v(x_P-x_Q\big)=$\\
$=v(\epsilon x_P+\epsilon x_Q+2x_Px_Q+2\epsilon x_Px_Q+x_P^2x_Q+x_Px_Q^2+2y_Py_Q)-2v(x_P-x_Q)\geq$
\vskip 0.2 cm
$\big($ since $2y_Py_Q=2\sqrt{x_Px_Q(x_P+\epsilon)(x_Q+\epsilon)(x_P+1)(x_Q+1)}<$\\
$<2\sqrt{x_Px_Q(2x_P)(2x_Q)(x_P+x_Q+1)^2}=4x_Px_Q(x_P+x_Q+1)$ $\big)$,
\vskip 0.2 cm
$\geq v(\epsilon(x_P x_Q+2x_Px_Q)+ x_Px_Q(x_P+x_Q+2)+4x_Px_Q(x_P+x_Q+1))-2v(x_P-x_Q)=$\\
$= v(\epsilon(x_P+x_Q+2x_Px_Q)+ x_Px_Q(5x_P+5x_Q+6))-2v(x_P-x_Q)=$
\vskip 0.2 cm
$\big($ since $v(\epsilon(x_P+x_Q+2x_Px_Q))=v(\epsilon)+\min\{v(x_P),v(x_Q)\}>v(x_P)+v(x_Q)=v(x_Px_Q(5x_P+5x_Q+6))$ $\big)$,
\vskip 0.2 cm
$=v(x_P)+v(x_Q)-2v(x_P-x_Q)$.

\vskip 0.2 cm
So $P\ominus Q\in G^{00}$ implies that $v(x_P)+v(x_Q)-2v(x_P-x_Q)\leq \frac{1}{n}v(\epsilon)$, for all $n$.
\vskip 0.2 cm

We recall that, since $sign{P}=sign{Q}$, $f_*(P)\cdot f_*(Q)^{-1}=\frac{x_Q}{x_P}$ or $f_*(P)\cdot f_*(Q)^{-1}=\frac{x_P}{x_Q}$, so $f_*(P)\cdot f_*(Q)^{-1}\in H^{00}$ if and only if $\left|v\left(\frac{x_Q}{x_P}\right)\right|\leq \frac{1}{n}v(\epsilon)$. 

\vskip 0.2 cm

We have two cases to consider:

\begin{itemize}
 \item 
If $x_P\geq x_Q$, then $v(x_P)\leq v(x_Q)$ and clearly $v\left(\frac{x_Q}{x_P}\right)\geq 0$, we just need to show that $ v\left(\frac{x_Q}{x_P}\right)\leq \frac{1}{n}v(\epsilon)$, for all $n$. But then $P\ominus Q\in G_{00}$ implies $v(x_P)+v(x_Q)-2v(x_P-x_Q)\geq v(x_P)+v(x_Q)-2v(x_P)=v\left(\frac{x_Q}{x_P}\right)$, so $v\left(\frac{x_Q}{x_P}\right)\leq \frac{1}{n}v(\epsilon)$, and $f_*(P)\cdot f_*(Q)^{-1}\in H^{00}$.

\item If $x_P< x_Q$, then $v(x_P)\geq v(x_Q)$, $v\left(\frac{x_Q}{x_P}\right)\leq 0$ and $\frac{1}{n}v(\epsilon)\geq v(x_P)+v(x_Q)-2v(x_P-x_Q)\geq v(x_P)+v(x_Q)-2v(x_Q)=v\left(\frac{x_P}{x_Q}\right)$, so $v\left(\frac{x_Q}{x_P}\right)\geq -\frac{1}{n}v(\epsilon)$, and we have proved the theorem.

\end{itemize}

\end{proof}

We can now easily check that $f$ is a bijection:
\begin{cor}\label{bijcor}
The map $f$ is a bijection $G/G^{00}\rightarrow H/H^{00}$.
\end{cor}
\begin{proof}
Surjectivity: trivial by construction.

\vskip 0.3 cm
Injectivity:  We need to consider only points of $E(K)^0$ not in $(T_2)_{\sim}$, $O_{\sim}$. Suppose $f(P_\sim)=f(Q_\sim)$. We have $\left|v\left(\frac{x_Q}{x_P}\right)\right| <\frac{1}{n}v(\epsilon)$, for all $n$. And by our assumption $0<x_P,x_Q<1$. We need to prove that $P\ominus Q\in O_{\sim}$, i.e., $v(x_{P\ominus  Q})<\frac{1}{n}v(\epsilon)$ for all $n$.

But $v(x_{P\ominus  Q})=$\\$=v(\epsilon x_P+\epsilon x_Q+2x_Px_Q+2\epsilon x_Px_Q+x_P^2x_Q+x_Px_Q^2+2y_Py_Q)-2v(x_P-x_Q)\leq$\\

$\big($  since $2y_Py_Q>2x_P^2x_Q^2$,$\big)$\\

$\leq  v(\epsilon(x_P x_Q+2x_Px_Q)+ x_Px_Q(x_P+x_Q+2)+2x_P^2x_Q^2)-2v(x_P-x_Q)=$\\
$= v(\epsilon(x_P+x_Q+2x_Px_Q)+ x_Px_Q(4x_P+4x_Q+1+x_Px_Q))-2v(x_P-x_Q)=$\\

(since $v(\epsilon(x_P+x_Q+2x_Px_Q))=v(\epsilon)+\min\{v(x_P),v(x_Q)\}>v(x_P)+v(x_Q)=v(x_Px_Q(4x_P+4x_Q+1+x_Px_Q))$,)\\

$=v(x_P)+v(x_Q)-2v(x_P-x_Q)\leq v(x_P)+v(x_Q)-2\min\{v(x_P),v(x_Q)\}$. 

But $v(x_P)+v(x_Q)-2\min\{v(x_P),v(x_Q)\}= \left|v\left(\frac{x_Q}{x_P}\right)\right| <\frac{1}{n}v(\epsilon)$ for all $n$, so also $v(x_{P\ominus  Q})<\frac{1}{n}v(\epsilon)$ for all $n$, and we are done.

\end{proof}

\begin{rmk}\label{remminus}
 From the proof above we deduce that if $P_{\sim}\neq Q_{\sim}$, and $P,Q$ are representatives in $E(K)^0$, then $v(x_{P\ominus Q})=v(x_P)+v(x_Q)-2\min\{v(x_P),v(x_Q)\}$, and a case-by-case study shows that $f:G/G^{00}\rightarrow H/H^{00}$ is an isomorphism. This is rather tedious and we omit the details. 
\end{rmk}

In \cite{Pen1} it is proved that the structure $(K,H^{00},\dots)^{eq}$ is interdefinable with a nonstandard real closed field $K_w^{eq}$, whose valuation is $w$ and that $H/H^{00}$ is a definable (in $K_w^{eq}$) group with underlying set in $\Gamma_w$.
Having found a definable bijection between $G/G^{00}$ and $H/H^{00}$, by Lemma \ref{coroneb},  we get the following theorem:

\begin{lem}\label{splmul}
 Given an elliptic curve $E$ with split multiplicative reduction, the group $G/G^{00}$ is $1$-based in the structure $K'=(K,G^{00},\dots)^{eq}$ and is in definable bijection with a group whose underlying set is in the value group $\Gamma_w$ of the real closed valued field interdefinable with $K'$.
\end{lem}

Lemma \ref{goodnonsp} and Lemma \ref{splmul}  prove part of Theorem \ref{thmpaper}. In the next section is proved the remaining part, with the analysis of the truncations.
\vskip 0.3 cm

\section{Truncations of elliptic curves}

Given an elliptic curve $E$ defined over a saturated real closed field $K$, a \emph{truncation} of $E(K)^0$is a group $G$ $\left([\ominus S,S),\oplus\mod[2]S\right)$, where $S\in E(K)^0\setminus T_2$, $y_S>0$, and the interval is considered according to the orientation $\lhd$ of $E(K)^0\setminus \{T_2\}$. We denote by $\oplus^*$ the operation on $G$. 

We now extend the classification above to such $G$ proving the following theorem: 
\begin{thm}\label{trunth}
The truncation $G=\left([\ominus S,S),\oplus\mod [2]S\right)$ of the $K$-points of an elliptic curve $E$ is $1$-based in $K'=(K,G^{00},\dots)^{eq}$ if and only if $G/G^{00}$ is in definable bijection with a group whose underlying set is in the value group of $K'=K_w^{eq}$, and if and only if $E$ has split multiplicative reduction and $v(x_S)> 0$.
\end{thm}

\begin{proof}

We shall consider all the possible cases, and therefore obtain all the implications in the theorem by exhaustion.
\vskip 0.2 cm

\begin{enumerate}
 \item
The first case is that of a truncation $G$ by a point $S\in E(K)^0\setminus E(K)^{00}$, then $G/G^{00}$ is simply a truncation of $ E(K)^0 / E(K)^{00}$ and thus $G/G^{00}$ has the same properties of $ E(K)^0 / E(K)^{00}$.

To see this, let $G=([\ominus S,S),\oplus \mod [2]S)$ and $S\notin E(K)^{00}$. This implies that $T^E_n\lhd P\lhd T^E_{n+1}$ for some $n$ and a bounding sequence $(T^E_n)_{n\in \mathbb{N}}$ of $E(K)^0$. For any $k$ let $T_k$ be a torsion point of a bounding sequence of $G$, defined as in Definition \ref{boundtors}, then it is easy to see that $x_{T^E_{kn}}<x_{T_k}<x_{T^E_{k(n+1)}}$, and therefore $G^{00}=E(K)^{00}$. Moreover $G/G^{00}$ is a definable truncation of $E(K)^{0}/E(K)^{00}$ in the expansion $K'$ of $K$ by a predicate for $G^{00}$, and so, by Corollary \ref{bijcor}, if $E$ has good or nonsplit multiplicative reduction, then $G/G^{00}$ is non-$1$-based in $K'$ and in definable bijection with a group with underlying set in the residue field of $K'$; if $E$ has split multiplicative reduction, $G/G^{00}$ is $1$-based and in definable bijection with a group with underlying set in the value group of $K'$.

\item This is the case of a truncation by a point $S$ such that $v(x_S)< 0$. 
 
Thus for $P\in G$, $v(x_P)<0$. Hence $v\left(x_{[2]P}\right)=v\left(\frac{(x_P^2-\epsilon)^2}{4x_P(x_P+1)(x_P+\epsilon)}\right)=2v(x_P^2-\epsilon)-3v(x_P)=v(x_P)$, and so $G^{00}=\left\{P\in G|v(x_P)<v(x_S)\right\}$.

It will suffice to prove that for $P,Q\notin G^{00}$ (and thus $v(x_Q)=v(x_P)=v(x_S)$), $P\ominus^* Q\in G^{00}$ (i.e. $v(x_{P\ominus^* Q})<v(x_S)$) if and only if $v(x_P-x_Q)>v(x_S)$ and $y_S,y_Q$ have the same sign. In fact this would imply that $G/G^{00}$ is in definable bijection with a definable group in the quotient $B_{\geq v(x_S)}(0)/B_{> v(x_S)}(0)$. We saw in Remark \ref{skelk} that there is a definable (in $K'$) field bijection $B_{\geq v(x_S)}(0)/B_{> v(x_S)}(0)\cong k_v\cong \mathbb{R}$, therefore $G/G^{00}$ is in definable bijection with a group with underlying set in the residue field of a real closed valued field and so it is non-$1$-based in $K'$ by Lemma \ref{coroneb}.

\vskip 0.2 cm

Suppose firstly that   $v(x_{P\ominus^* Q})<v(x_S)$.

 Using the computation in Theorem \ref{defbij}, $v(x_{P\ominus^* Q})\geq v(\epsilon(x_P+x_Q+2x_Px_Q)+x_Px_Q(5x_P+5x_Q+6))-2v(x_P-x_Q)=$

(since $v(x_Q),v(x_P)= v(x_S)<0\leq v(\epsilon)$), 

$=v(x_P)+v(x_Q)+\min\{v(x_P),v(x_Q)\}-2v(x_P-x_Q)$. 

So $2v(x_P-x_Q)> 2v(x_S)$, so $v(x_P-x_Q)>v(x_S)$),
\vskip 0.3 cm

Now suppose $v(x_P-x_Q)>v(x_S)$.   Then 
$v\left(x_{P\ominus^* Q}\right)=$\\
$=v(\epsilon x_P+\epsilon x_Q+2x_Px_Q+2\epsilon x_Px_Q+x_P^2x_Q+x_Px_Q^2+2y_Py_Q)-2v(x_P-x_Q)\leq$

(since $2y_Py_Q>2x_Px_Q$),

$\leq v(\epsilon(x_P x_Q+2x_Px_Q)+ x_Px_Q(x_P+x_Q+2)+x_Px_Q(4x_P+4x_Q+6))=$\\
$= v(\epsilon(x_P+x_Q+2x_Px_Q)+ x_Px_Q(5x_P+5x_Q+6))-2v(x_P-x_Q)=$

$=v(x_P)+v(x_Q)+\min\{v(x_P),v(x_Q)\}-2v(x_P-x_Q)\leq v(x_P)+v(x_Q)+\min\{v(x_P),v(x_Q)\}-2v(x_S)=3v(x_S)-2v(x_S)=v(x_S)$.

With this we proved Case 2.

\end{enumerate}

The above are the only possible cases when $E$ has good or nonsplit multiplicative reduction. We have two more cases when $E$ has split multiplicative reduction. So from now on we assume $v(\epsilon)>0$.

\begin{enumerate}  
 \item[3.] $S\in E(K)^{00}$ and $v(x_S)> 0$. With such assumptions any point $P\in G$ has valuation $v(x_P)<v(\epsilon)$. Then $v\left(x_{[2]P}\right)=v\left(\frac{(x_P^2-\epsilon)^2}{4x_P(x_P+1)(x_P+\epsilon)}\right)=2v(x_P^2+\epsilon)-v(x_P)-0-v(x_P)=2v(x_P)$. Thus $G^{00}=\{P\in G| v(x_P)<\frac{1}{n}v(\epsilon)\}$.

As in the split multiplicative case we can define in the suitable expansion a bijection $G/G^{00}\rightarrow H/H^{00}$ with $H=\left(\left[x_S,\frac{1}{x_S}\right),*\mod \left(\frac{1}{x_S}\right)^2\right)$ a ``big'' multiplicative truncation.

The map $f_*: G\rightarrow H$ defined by \[ f_*(P)=\left\{\begin{array}{cl}
               1& ~if ~ x_P\geq 1\\
\\
               \left(\frac{1}{x_P}\right)& ~if~ y_P\geq0\wedge x_P<1, \\
\\
 x_P& ~if~ y_P<0\wedge x_P<1,\\

              \end{array}\right.\]

induces a map $f:G/G^{00}\rightarrow H/H^{00}$. The same calculation that led to Corollary \ref{bijcor} gives us that $f$ is a definable bijection. Therefore $G/G^{00}$ inherits $1$-basedness from $H/H^{00}$ by Lemma \ref{coroneb} and again it is in definable bijection with a group with underlying set in the value group of a real closed valued field.

\item[4.] $S\in E(K)^{00}$ and $v(x_S)= 0$. It is again immediate to observe that if $x_P\in G$ and $v(x_P)=0$, $v\left(x_{[2]P}\right)=2v(x_P)$. Therefore $G^{00}=\{P\in G|v(x_P)<0\}$. By the same argument as Case 3 we obtain a definable bijection with a multiplicative truncation, though this time it is a ``small'' one, and therefore $G/G^{00}$ is in definable bijection with a group with underlying set in the residue field of a real closed valued field and, again by Lemma \ref{coroneb}, non-$1$-based in $K'$.
\end{enumerate}

The inspection of the cases considered gives us the proof of Theorem \ref{trunth}.

\end{proof}

With this last case study we have completed the proof of Theorem \ref{trunth} and therefore of Theorem \ref{thmpaper}.

\vskip 1 cm

It is a natural question now to ask to what extent the notion of ``intrinsic'' reduction can help in obtaining a reduction theory for abelian varieties over fields with a continuous valuation. In particular we wonder whether we can obtain a similar classification of higher dimensional abelian varieties.

\vskip 4 cm

The author would like to thank Prof. Anand Pillay for his guidance and support, Dr. Marcus Tressl for many interesting discussions and the anonymous referee for the many good suggestions.

\vskip 2 cm
   Davide Penazzi\\
   School of Mathematics, University of Leeds, Woodhouse Lane, LS2 9JT,\\
   UK
\end{document}